\documentclass[12 pt]{amsart}
\usepackage{amssymb, amsmath, amsfonts, amsthm, graphics,mathrsfs}
\usepackage[hmargin=1 in, vmargin = 1 in]{geometry}
\usepackage{tikz-cd} 
\usetikzlibrary{matrix, calc, arrows} 
\usepackage{hyperref}
\usepackage[all]{xy}

\newcommand{\newword}[1]{\textbf{\emph{#1}}}

\newcommand{\GL}{\mathrm{GL}}

\newcommand{\CC}{\mathbb{C}}

\newcommand{\ZZ}{\mathbb{Z}}

\newcommand{\cS}{\mathcal{S}}

\renewcommand{\sl}{\mathfrak{sl}}
\newcommand{\SL}{\mathrm{SL}}
\newcommand{\mS}{\mathfrak{S}}
\newcommand{\Id}{\mathrm{Id}}

\newcommand{\der}[1]{\tfrac{\partial}{\partial x_{#1}}}
\newcommand{\diff}{N}

\newcommand{\poly}{\ensuremath{\mathrm{Poly}}}
\newcommand{\W}{W}

\newtheorem{Theorem}{Theorem}

\newtheorem{cor}[Theorem]{Corollary}

\newtheorem{proposition}[Theorem]{Proposition}

\newtheorem{theorem}[Theorem]{Theorem}

\newtheorem{lemma}[Theorem]{Lemma}

\begin{document}

\title[Derivatives of Schubert Polynomials]{Derivatives of Schubert Polynomials and Proof of a Determinant Conjecture of Stanley}
\author{Zachary Hamaker, Oliver Pechenik, David E Speyer, Anna Weigandt}
\address{Department of Mathematics, University of Michigan, Ann Arbor, MI 48109}
\email{$\{$hamaker,pechenik,speyer,weigandt$\}$@umich.edu}
\date{\today}

\begin{abstract}
We study the action of a differential operator on Schubert polynomials. Using this action, we first give a short new proof of an identity of I.~Macdonald (1991). We then prove a determinant conjecture of R.~Stanley (2017). This conjecture implies the (strong) Sperner property for the weak order on the symmetric group, a property recently established by C.~Gaetz and Y.~Gao (2018).
\end{abstract}

\maketitle

\section{Introduction}

This paper is motivated by a conjecture of R.~Stanley~\cite[Conjecture~2.2]{Stanley:shen}. 
Let $\cS_n$ be the symmetric group with its standard generating set $S = \{s_1, s_2, \dots, s_{n-1}\}$, and let $\cS_n(\ell)$ denote the subset of those permutations of (Coxeter) length $\ell$. 
For $1 \leq \ell \leq \binom{n}{2}$, let $M_{\ell}$ be the matrix with rows indexed by $\cS_n(\ell-1)$ and columns indexed by $\cS_n(\ell)$, where the entry in position $(u,v)$ is 
\[
M_\ell [u,v] = 
\begin{cases}
k, &\text{if}\enspace v=us_k \enspace\text{and}\\
0, &\text{if $u^{-1} v \notin S$}.
\end{cases}
\]

For $\ell \leq \binom{n}{2}-\ell$, the product $\widetilde{M}^{(\ell)} = M_{\ell+1} M_{\ell+2} \cdots M_{\binom{n}{2} - \ell}$ is a square matrix with rows indexed by $\cS_n(\ell)$ and columns indexed by $\cS_n \left(\binom{n}{2} - \ell \right)$. Stanley conjectures an explicit formula for $\det \widetilde{M}^{(\ell)}$,  which implies that $\widetilde{M}^{(\ell)}$ is invertible.
A motivation for this conjecture is that $M_{\ell(u)+1} M_{\ell(u)+2} \cdots M_{\ell(v)}[u,v]$ is nonzero if and only if $u \leq v$ in weak order.
Hence, by standard linear-algebraic arguments (cf.~\cite{Stanley:Weyl, Stanley:shen}), showing that $\det \widetilde{M}^{(\ell)}$ is nonzero implies that the weak order on the symmetric group has the (strong)  Sperner property.

Recently, C.~Gaetz and Y.~Gao \cite{Gaetz.Gao} proved the invertibility of Stanley's matrix by constructing an action of the Lie algebra $\sl_2$.
We give a new proof of invertibility by proving Stanley's determinant conjecture.
Our proof also involves an $\sl_2$-representation, but relies on a new identity for derivatives of Schubert polynomials. 

The generator $s_i$ acts on $\poly_n = \CC[x_1, \ldots, x_n]$ by 
\[
s_i \cdot f(x_1, \ldots, x_i, x_{i+1}, \ldots, x_n) = f(x_1, \ldots, x_{i+1}, x_{i}, \ldots, x_n).
\]
We define the Newton divided difference operators on $\poly_n$ by
\[ \diff_i(f) = \frac{f - s_i \cdot f}{x_{i} - x_{i+1}}. \]
(We avoid the more standard notation $\partial_i$ because of potential confusion with partial derivatives.)
For $w \in \cS_n$, the \newword{Schubert polynomials} $\mS_w$ are defined by the recurrence
\[ \mS_{s_k w} = \diff_k \mS_w \ \mbox{for} \ \ell(s_k w) < \ell(w) \] with $\mS_{w_0} = x_1^{n-1} x_2^{n-2} \cdots x_{n-1}$.
For background on Schubert polynomials, we refer the reader to \cite{Macdonald:notes,Manivel}.

We consider the differential operator $\nabla = \sum_{i=1}^n \der{i}$. Our key result is the following:
\begin{proposition} \label{DerIdentity}
For $w \in \cS_n$, we have
\[ \nabla (\mS_w) = \sum_{\ell(w s_k) < \ell(w)} k \mS_{ws_k} . \]
\end{proposition}
Thus, Stanley's $M$-matrices are the matrices of the operator $\nabla$ in the basis of Schubert polynomials. 

We first apply Proposition~\ref{DerIdentity} to give a short new proof of a theorem of Macdonald~\cite[(6.11)]{Macdonald:notes}.
We then use Proposition~\ref{DerIdentity} to prove Stanley's conjecture. 
\begin{theorem}[{Conjectured by Stanley~\cite[Conjecture~2.2]{Stanley:shen}}] \label{MainTheorem}
For $\ell \leq \binom{n}{2}-\ell$, 
\[ \det \widetilde{M}^{(\ell)}  =  \pm \prod_{k=0}^{\ell} \left( (\ell-k+1) (\ell-k+2) \cdots ( \textstyle{\binom{n}{2}} - \ell -k)   \right)^{|\cS_n(k)| - |\cS_n(k-1)| }. \]
\end{theorem}
We write $\pm$ because we have not specified an order on the rows and columns of each $M_i$.

\section{Proof of Proposition~\ref{DerIdentity} and a Macdonald identity}

We start with a straightforward lemma.
\begin{lemma}\label{lem:commutation}
$\nabla$ commutes with $\diff_i$ for all $i$. That is, for any $f \in \poly_n$, we have
\[
\nabla(\diff_i(f)) = \diff_i(\nabla(f)).
\]
\end{lemma}
\begin{proof}
If $j \notin \{ i, i+1 \}$, then \[
\der{j} \diff_i(f) = 
\frac{1}{x_i - x_{i+1}} \left( \der{j} (f) - \der{j} (s_i \cdot f) \right).
\]

If $j = i$, then 
\[
\der{i} \diff_i(f) = \frac{1}{x_i - x_{i+1}} \left( \der{i} (f) - \der{i} (s_i \cdot f) \right) - \frac{f - s_i \cdot f}{(x_i - x_{i+1})^2}.
\]

Similarly, if $j = i+1$, then
\[
\der{i+1} \diff_i(f) 
= \frac{1}{x_i - x_{i+1}} \left( \der{i+1} (f) - \der{i+1} (s_i \cdot f) \right) + \frac{f - s_i \cdot f}{(x_i - x_{i+1})^2}.
\]
Therefore, 
\begin{align*}
\nabla(\diff_i(f)) &= \frac{f - s_i \cdot f}{(x_i - x_{i+1})^2} - \frac{f - s_i \cdot f}{(x_i - x_{i+1})^2} +   \sum_{j=1}^n \frac{1}{x_i - x_{i+1}} \left( \der{j} (f) - \der{j} (s_i \cdot f) \right) \\
&= \frac{\sum_{j=1}^n \der{j}(f) - \sum_{j=1}^n \der{j}(s_i \cdot f)}{x_i - x_{i+1}}  = \frac{\nabla(f) - \nabla(s_i \cdot f)}{x_i - x_{i+1}} \\
&= \frac{\nabla(f) - s_i \cdot \nabla(f)}{x_i - x_{i+1}} = \diff_i (\nabla (f)),
\end{align*}
as desired.
\end{proof}

\begin{proof}[Proof of Proposition~\ref{DerIdentity}]
We first verify Proposition~\ref{DerIdentity} in the case that $w = w_0$. Since $\mS_{w_0} = x_1^{n-1} x_2^{n-2} \cdots x_{n-1}$,  we have
\begin{align*}
\nabla(\mS_{w_0}) &= \sum_{j=1}^{n-1} (n-j) x_1^{n-1} x_2^{n-2} \cdots x_{j-1}^{n-j+1} x_j^{n-j-1} x_{j+1}^{n-j-1} \cdots x_{n-1} \\
&= \sum_{k=1}^{n-1} k x_1^{n-1} x_2^{n-2} \cdots x_{n-k-1}^{k+1} x_{n-k}^{k-1} x_{n-k+1}^{k-1} \cdots x_{n-1} . 
\end{align*}
 But also
\[ \mS_{w_0 s_{k}} = \mS_{s_{n-k} w_0} = \diff_{n-k}(\mS_{w_0}) =  x_1^{n-1} x_2^{n-2} \cdots x_{n-k-1}^{k+1} x_{n-k}^{k-1} x_{n-k+1}^{k-1} \cdots x_{n-1} . \]
Comparing these equations gives
\[ \nabla(\mS_{w_0}) = \sum_{k=1}^{n-1} k \mS_{w_0 s_k} . \]

Consider an arbitrary permutation $w$. Let $r = \binom{n}{2} - \ell(w)$ and write $w = s_{i_1} s_{i_2} \cdots s_{i_r} w_0$.
By Lemma~\ref{lem:commutation}, we have
\begin{equation}\label{eq:nabla}\tag{$\star$}
 \nabla(\mS_{w}) = \nabla \diff_{i_1} \diff_{i_2} \cdots \diff_{i_r}(\mS_{w_0}) = \diff_{i_1} \diff_{i_2} \cdots \diff_{i_r} \nabla(\mS_{w_0}) =  \diff_{i_1} \diff_{i_2} \cdots \diff_{i_r}  \sum_{k=1}^{n-1} k \mS_{w_0 s_k} . 
 \end{equation}
Hence, \[
\diff_{i_1} \diff_{i_2} \cdots \diff_{i_r} ( \mS_{w_0 s_k} )= 
\begin{cases}
\mS_{s_{i_1} \cdots s_{i_r} w_0 s_k}, &\text{if $\ell(w s_k) = \binom{n}{2} - r -1 = \ell(w)-1$} \\
0, & \text{otherwise}.
\end{cases}
\] 
Since $\mS_{s_{i_1} \cdots s_{i_r} w_0 s_k} = \mS_{w s_k}$, Equation~\eqref{eq:nabla} then becomes
\[ \nabla (\mS_{w}) = \sum_{\ell(w s_k) = \ell(w)-1} k \mS_{w s_k}, \]
as desired.
\end{proof}

Proposition~\ref{DerIdentity} yields a short proof of an identity of Macdonald~\cite[(6.11)]{Macdonald:notes}. Another proof of this result was given by S.~Fomin and R.~Stanley in terms of nilCoxeter algebras \cite{Fomin.Stanley}, while a bijective proof was given by S.~Billey, A.~Holroyd and B.~Young \cite{Billey.Holroyd.Young}. A \newword{reduced word} for $w \in \cS_n$ is a tuple $a = (a_1, \dots, a_{\ell(w)})$ such that $w=s_{a_1} s_{a_2} \cdots s_{a_{\ell(w)}}$. We write $R(w)$ for the set of reduced words of $w$.

\begin{theorem}[{Macdonald~\cite[(6.11)]{Macdonald:notes}}]
Let $w \in \cS_n$ with $\ell(w) = k$. Then
\[ \frac{1}{k! } \sum_{a \in R(w)} a_1 a_2 \cdots a_k = \mS_w(1,1,\ldots,1). \]
\end{theorem}

\begin{proof}
For any monomial $\mu$ of degree $k$, we have $\nabla^k (\mu) = k!$.
Since $\mS_w$ is homogenous of degree $k$, we then see $\nabla^k( \mS_w) = k! \mS_w(1,1,\ldots,1)$.
On the other hand, by Proposition~\ref{DerIdentity}, 
\[ \pushQED{\qed} \nabla^k (\mS_w) =  \sum_{a \in R(w)} a_1 a_2 \cdots a_k. \qedhere \popQED \] \let\qed\relax
\end{proof}

\section{A vector space of polynomials}\label{sec:vectorspace}

Let $\W \subset \poly_n$ be the span of the monomials of the form $x_1^{a_1} x_2^{a_2} \cdots x_n^{a_n}$ with $0 \leq a_j \leq n-j$. 
(In particular, the exponent of $x_n$ is required to be $0$, so the variable $x_n$ does not occur in any polynomial in $\W$.)
Let $\W_{\ell}$ be the subspace of $\W$ spanned by monomials of degree $\ell$. We will need the following lemma.

\begin{lemma}[{\cite[Proof of Corollary~3.9]{Bergeron.Billey}}] \label{LeadingTerm}
For $w \in \cS_n$, the Schubert polynomial $\mS_w$ lies in $\W$. If we choose a term order with $x_n > x_{n-1} > \cdots > x_1$, then the leading term of $\mS_w$ is
\[ \pushQED{\qed} \prod_{j=1}^n x_j^{\# \{ k \ : \ k>j,\ w(k) < w(j) \} }.  \qedhere \popQED \] 
\end{lemma}

The list of numbers $\# \{ k \ : \ k>j,\ w(k) < w(j) \}$ is the \newword{(Lehmer) code} of $w$; taking the code is a bijection between $\cS_n$ and $\{ (a_1, \ldots, a_n) \in \ZZ^n : 0 \leq a_j \leq n-j \}$ (see, e.g., \cite[Proposition~2.1.2]{Manivel}). Thus, Lemma~\ref{LeadingTerm} implies that the Schubert polynomials have distinct leading terms and we deduce:
\begin{cor} \label{ChangeBasis}
The Schubert polynomials $\mS_w$ for $w \in \cS_n$ are a basis for $\W$. The change of basis matrix between $\{ \mS_w : \ell(w) = k \}$ and $\{ x_1^{a_1} x_2^{a_2} \cdots x_n^{a_n} \ : \ 0 \leq a_j \leq n-j,\ \sum a_j =k \}$ has determinant $\pm 1$. \qed
\end{cor}

Proposition~\ref{DerIdentity} shows that $\nabla : \W_{\ell} \to \W_{\ell-1}$, in the Schubert basis, is represented by the matrix $M_{\ell}$. Therefore, to prove Theorem~\ref{MainTheorem}, we must compute 
\[ \det\left( \W_{\binom{n}{2} - \ell} \xrightarrow{\ \nabla^{\binom{n}{2} - 2 \ell}\ } \W_{\ell} \right) \]
in the Schubert basis. By Corollary~\ref{ChangeBasis}, we may compute this determinant instead in the monomial basis. 
For the remainder of this note, Schubert polynomials disappear and our goal is to compute the determinant of $\nabla^j$ acting with respect to the monomial basis.

We would prefer to have a map from a vector space to itself, so that we could speak of its determinant without any reference to bases. 
There is a simple bijection between the monomial bases of $\W_k$ and $\W_{\binom{n}{2} - k}$, taking $\prod x_j^{a_j}$ to $\prod x_j^{n-j-a_j}$. 
For reasons that will become clear in Section~\ref{sec:sl2_reps}, we prefer to twist this map by $(-1)^k$, so we define $J$ to be the linear endomorphism of $\W$ with
\[ J\left( \prod x_j^{a_j} \right) = (-1)^{\sum a_j} \prod x_j^{n-j-a_j} . \]
Note that $J$ has determinant $\pm 1$ in the monomial basis and that $\nabla^{\binom{n}{2} - 2\ell} \circ J$ maps $\W_{\ell}$ to itself. Thus, to finish our proof of Theorem~\ref{MainTheorem}, it remains to establish
\[ \det \left( \W_{\ell} \xrightarrow{\ \nabla^{\binom{n}{2} - 2\ell} \circ J \ } \W_{\ell} \right) =   \pm \prod_{k=0}^{\ell} \left( (\ell-k+1) (\ell-k+2) \cdots ( \textstyle{\binom{n}{2}} - \ell -k)   \right)^{|\cS_n(k)| - |\cS_n(k-1)| },\]
a statement that makes no reference to bases. We now turn to this task.

\section{$SL_2$-representations and a proof of Theorem~\ref{MainTheorem}}\label{sec:sl2_reps}

In this section, we discuss some representations of the Lie group $\SL_2$ and its Lie algebra $\sl_2$. We will denote group actions by variants of the letter $\rho$ and the 
corresponding Lie algebra actions by variants of $\sigma$. We write the standard basis of $\sl_2$ as
\[ F = \begin{bmatrix} 0&0\\1&0 \end{bmatrix} \qquad H = \begin{bmatrix} 1&0\\0&-1 \end{bmatrix} \qquad E = \begin{bmatrix} 0&1\\0&0 \end{bmatrix} \]
and define the element
\[ J = \begin{bmatrix} 0&1 \\ -1 & 0 \end{bmatrix} \in \SL_2.\]

Let $V_k$ be the $(k+1)$-dimensional irreducible representation of $\SL_2$ and $\sl_2$; we write $\rho_k$ and $\sigma_k$ for the action maps $\rho_k : \SL_2 \to \GL(V_k)$ and $\sigma_k : \sl_2 \to \mathrm{End}(V_k)$. 

One usually describes $V_k$ as the natural action on degree $k$ polynomials in two variables.
For our purposes, it is more convenient to describe $V_k$ as an action on polynomials  of degree $\leq k$ in one variable $x$.
We have
\begin{equation}\tag{$\dagger$}\label{eq:rep}
\begin{array}{rcr@{}l}
\sigma_k(F)(x^j) &=& j &x^{j-1} \\  \sigma_k(H)(x^j) &=& (2j-k) &x^j \\  \sigma_k(E) (x^j) &=& (k-j) &x^{j+1} \\  \rho_k(J)(x^j) &=& (-1)^j &x^{k-j} . \\
\end{array}  
\end{equation}
In particular, for any polynomial $f$, we have $\sigma_k(F)(f) = \tfrac{d f}{d x}$. 

Identify the vector space $\W$ from Section~\ref{sec:vectorspace} with $V_{n-1} \otimes V_{n-2} \otimes \cdots  \otimes V_0$ by identifying $x_1^{a_1} x_2^{a_2} \cdots x_{n}^{a_{n}}$ with $x^{a_1} \otimes x^{a_2} \otimes \cdots \otimes x^{a_{n}}$. 
We let $\SL_2$ and $\sl_2$ act on this tensor product in the standard way, and denote these actions by $\sigma_\W$ and $\rho_\W$. 
We note that $\W_{\ell}$ is the $2 \ell-\binom{n}{2}$ weight space, i.e.,\ the $2 \ell-\binom{n}{2}$ eigenspace of $\sigma_\W(H)$.

We have
\[ \sigma_\W(F) = \sum_{k=1}^n \Id \otimes \Id \otimes \cdots \otimes \sigma_{n-k}(F) \otimes \cdots \otimes \Id \]
where $\sigma_{n-k}(F)$ occurs in the $k$-th position. Therefore,
\[ \sigma_\W(F) \cdot f = \sum_{k=1}^n \der{k} f  = \nabla f  .\]
Similarly, 
\[ \rho_\W(J) = \rho_{n-1} \otimes \rho_{n-2} \otimes \cdots \otimes \rho_0, \]
so
\[ 
\rho_\W(J) \cdot \left( \prod_{j=1}^n x_j^{a_j} \right) = \prod_{j=1}^n (-1)^{a_j} x_j^{n-j-a_j} = (-1)^{\sum a_j} \prod_{j=1}^n  x_j^{n-j-a_j} = J \left( \prod_{j=1}^n x_j^{a_j} \right).
\]

Thus, our goal of computing $\det \left( \nabla^{\binom{n}{2} - 2 \ell} \circ J \right)$ on $V_{\ell}$ is the same as computing the determinant of $\sigma_\W(F)^{\binom{n}{2} - 2 \ell} \rho_\W(J)$ as a map from the $2 \ell - \binom{n}{2}$ weight space of $V_{n-1} \otimes V_{n-2} \otimes \cdots \otimes V_0$ to itself.

This is a standard computation. By comparing dimensions of weight spaces, 
\[ V_{n-1} \otimes V_{n-2} \otimes \cdots \otimes V_0 \ \cong \ \bigoplus_{0 \leq k \leq \binom{n}{2} - k} V_{\binom{n}{2}- 2 k}^{\oplus \, |\cS_n(k)| - |\cS_n(k-1)|} . \]
Thus, Theorem~\ref{MainTheorem} comes down to showing that $\sigma_\W(F)^{\binom{n}{2} - 2 \ell} \rho_\W(J)$ as a map from the $2 \ell - \binom{n}{2}$ weight space of $V_{\binom{n}{2}- 2 k}$ to itself is $\pm (\ell-k+1) (\ell-k+2) \cdots ( \textstyle{\binom{n}{2}} - \ell -k)$.
Consulting the formulas from Equation~\eqref{eq:rep}, one sees that
\[ \rho_{\binom{n}{2} - 2k}(J) \ x^{\ell-k} = (-1)^{\ell-k} x^{\binom{n}{2} -2k -(\ell-k)} = (-1)^{\ell-k} x^{\binom{n}{2} - \ell - k }\]
and
\[ \sigma_{\binom{n}{2}-2k}(F)^{\binom{n}{2} - 2 \ell} x^{\binom{n}{2} - \ell - k } = \left(\tfrac{d}{dx} \right)^{\binom{n}{2}-2 \ell}  x^{\binom{n}{2} - \ell - k } = (\ell-k+1) (\ell-k+2) \cdots ( \textstyle{\binom{n}{2}} - \ell -k) x^{\ell-k} .\]
Theorem~\ref{MainTheorem} follows. \qed


\section*{Acknowledgments}
We are grateful for helpful conversations with Ben Elias, Sergey Fomin and Benjamin Young.

OP was partially supported by a Mathematical Sciences Postdoctoral Research Fellowship (\#1703696) from the National Science
Foundation.
DES was partially supported by NSF Grant DMS-1600223. 

\bibliographystyle{amsalpha} 
\bibliography{MacdonaldId}

\providecommand{\bysame}{\leavevmode\hbox to3em{\hrulefill}\thinspace}
\providecommand{\MR}{\relax\ifhmode\unskip\space\fi MR }
\providecommand{\MRhref}[2]{%
  \href{http://www.ams.org/mathscinet-getitem?mr=#1}{#2}
}
\providecommand{\href}[2]{#2}
\begin{thebibliography}{BHY18}

\bibitem[BB93]{Bergeron.Billey}
Nantel Bergeron and Sara Billey, \emph{R{C}-graphs and {S}chubert polynomials},
  Experiment. Math. \textbf{2} (1993), no.~4, 257--269.

\bibitem[BHY18]{Billey.Holroyd.Young}
Sara~C Billey, Alexander~E Holroyd, and Benjamin Young, \emph{A bijective proof
  of {M}acdonald's reduced word formula}, Algebraic Combin. (2018), 39 pages,
  to appear, {\sf arXiv:1702.02936}.

\bibitem[FS94]{Fomin.Stanley}
Sergey Fomin and Richard~P. Stanley, \emph{Schubert polynomials and the
  nil-{C}oxeter algebra}, Adv. Math. \textbf{103} (1994), no.~2, 196--207.

\bibitem[GG18]{Gaetz.Gao}
Christian Gaetz and Yibo Gao, \emph{A combinatorial {$\mathfrak{sl}_2$}-action
  and the {S}perner property for the weak order}, preprint (2018), 6 pages,
  {\sf arXiv:1811.05501}.

\bibitem[Mac91]{Macdonald:notes}
I.~G. Macdonald, \emph{Notes on {S}chubert polynomials}, vol.~6, Publications
  du LACIM, Universit\'{e} du Quebec \`{a} Montreal, 1991.

\bibitem[Man01]{Manivel}
Laurent Manivel, \emph{Symmetric functions, {S}chubert polynomials and
  degeneracy loci}, SMF/AMS Texts and Monographs, vol.~6, American Mathematical
  Society, Providence, RI; Soci\'{e}t\'{e} Math\'{e}matique de France, Paris,
  2001, Translated from the 1998 French original by John R. Swallow, Cours
  Sp\'{e}cialis\'{e}s [Specialized Courses], 3.

\bibitem[Sta80]{Stanley:Weyl}
Richard~P. Stanley, \emph{Weyl groups, the hard {L}efschetz theorem, and the
  {S}perner property}, SIAM J. Algebraic Discrete Methods \textbf{1} (1980),
  no.~2, 168--184.

\bibitem[Sta17]{Stanley:shen}
\bysame, \emph{Some {S}chubert shenanigans}, preprint (2017), 8 pages, {\sf
  arXiv:1704.00851}.

\end{thebibliography}

\end{document}